\documentclass[11pt,a4paper]{article}
\usepackage[utf8]{inputenc}
\usepackage{bold-extra}
\usepackage{chngcntr}
\usepackage{tikz}
\usepackage{todonotes}
\usepackage{setspace}
\usepackage{url}
\usepackage{amsfonts}
\usepackage{epsfig}
\usepackage{commath}
\usepackage{natbib}
\usepackage{graphicx}
\usepackage{bbm,float}
\usepackage{dsfont}
\usepackage{color}
\usepackage{adjustbox}
\usepackage{amsmath,amssymb,amstext,amsfonts,amsthm,latexsym}
\usepackage{diagbox}
\usepackage{bm}
\usepackage{caption}
\usepackage{subcaption}
\usepackage{multirow}
\usepackage{color}
\usepackage{listings}
\usepackage{titlesec}
\usepackage{enumerate}
\usepackage{environ}
\usepackage{booktabs}
\usepackage{appendix}
\usepackage{bbm}
\usepackage{pgfplots}
\usepackage{makecell}
\usepackage{setspace}
\pgfplotsset{compat=1.16}
\numberwithin{equation}{section}
\theoremstyle{plain}
\newtheorem{theorem}{Theorem}[section]
\newtheorem{corollary}{Corollary}[section]
\newtheorem{proposition}{Proposition}[section]

\newtheorem{lemma}{Lemma}[section]

\theoremstyle{definition}

\newtheorem{definition}{Definition}[section]

\newtheorem{algorithm}{Algorithm}[section]

\theoremstyle{remark}

\newtheorem{remark}{Remark}[section]
\usepackage{setspace}
\spacing{1.2}
\begin{document}
\title{\textbf{Coskewness under dependence uncertainty}
} 
\date{\today}
\author{\textit{Carole Bernard\thanks{Carole Bernard, Department of Accounting, Law and Finance, Grenoble Ecole de Management (GEM) and Department of Business at Vrije Universiteit Brussel (VUB). (email: \texttt{carole.bernard@grenoble-em.com}).}, Jinghui Chen\thanks{Corresponding author: Jinghui Chen, Department of Business at Vrije Universiteit Brussel (VUB). (email: \texttt{jinghui.chen@vub.be}).}, Ludger R\"{u}schendorf\thanks{Ludger R\"{u}schendorf, Department of Stochastic at University of Freiburg. (email: \texttt{ruschen@stochastik.uni-freiburg.de}).} and Steven Vanduffel\thanks{Steven Vanduffel, Department of Business at Vrije Universiteit Brussel (VUB). (email: \texttt{steven.vanduffel@vub.be}).}}}
\maketitle
\begin{onehalfspacing}
\begin{abstract}
	We study the impact of dependence uncertainty on 
	$\mathbb{E}(X_1X_2\cdots X_d)$ when $X_i \sim F_i$ for all~$i$. 
	Under some conditions on the $F_i$, explicit sharp bounds are obtained and  
	a numerical method is provided to approximate them for arbitrary choices of  the $F_i$. 
	The results are applied to assess the impact of dependence uncertainty on coskewness. In this regard, we introduce a novel notion of ``standardized rank coskewness," which is invariant under strictly increasing transformations and takes values in $[-1,\ 1]$.
\end{abstract}

\end{onehalfspacing}

\thispagestyle{empty}
\newpage
 \pagenumbering{arabic}
\section{Introduction}
A fundamental characteristic of a multivariate random vector $\left(X_1,X_2,\dots, X_d\right)$ concerns the k-th order mixed moment
\begin{equation*}
	\mathbb{E}\left(X_1^{k_1}X_2^{k_2}\cdots X_d^{k_d}\right)=\int_{\mathbb{R}^d}x_1^{k_1}x_2^{k_2}\cdots x_d^{k_d}dF(x_1, x_2, \dots, x_d),
\end{equation*}
where $k_i,$ $i=1,2,\dots,d,$ are non-negative integers such that $\sum_{i=1}^{d}k_i=k$ and $F$ is the joint distribution function of $\left(X_1,X_2,\dots, X_d\right)$; see e.g., \cite{kotz2004continuous}. A classic problem in multivariate modeling is to find sharp bounds on mixed moments 
under the assumption that the marginal distribution functions of the $X_i$ are known but not their dependence. The solutions in case $d=2$ are well-known but for higher dimensions a complete solution is still missing. In this regard, it is well-known  that under the assumption that all $X_i$ are non-negative, the sharp upper bound is obtained in case the variables have a comonotonic dependence. As for the lower bound problem, \cite{wang2011complete} obtain a sharp bound under the assumption that the $X_i$ are standard uniformly distributed. To the best of our knowledge, there are no other relevant results available in the literature. 

In the first part of this paper, we determine for the case $k_i=1$ $(i=1,2,\dots,d)$ sharp lower and upper bounds on mixed moments under some assumptions on the marginal distribution functions of $X_i$. When $k_i\neq 1$ and the domain of marginal distributions is non-negative, these bounds are also solvable under a mixing assumption on the distributions of $k_i\ln X_i$. Furthermore, we establish a necessary condition that solutions to the optimization problems need to satisfy and use this result to design an algorithm that approximates the sharp bounds.   

A special case of finding bounds on mixed moments concerns the case of standardized central mixed moments, such as covariance (second-order), coskewness (third-order) and cokurtosis (fourth-order). The sharp lower and upper bounds in the case of covariance are very well-known in the literature, and in the second part of the paper we focus on the application of our results to obtaining bounds on coskewness. We obtain explicit risk bounds for some popular families of marginal distributions, such as uniform, normal and Student's t distributions. 
Furthermore, we introduce the novel notion of  standardized rank coskewness and discuss its properties. Specifically, as the standardized rank coskewness takes values in $[-1,\ 1]$ and is not affected by the choice of marginal distributions, this notion makes it possible to interpret the sign and magnitude of coskewness without impact of marginal distributions. In spirit, the standardized rank coskewness extends the notion of Spearman's correlation coefficient to three dimensions. 

The paper is organized as follows. In Section \ref{problem setting}, we lay out the optimization problem. 
In Section~\ref{main theorem}, we derive  sharp bounds under various conditions on the marginal distribution functions and also provide a numerical approach to approximate the sharp bounds in general.  
We apply our results in Section \ref{uniform margins} to  introduce the notion of standardized rank coskewness. 

\section{Problem setting} \label{problem setting} 
In what follows all random variables $X_i\sim F_i$, $i=1,2,\dots,d$, that we consider are assumed to be square integrable. We denote their means and standard deviations by $\mu_i$ and $\sigma_i$, respectively. Furthermore, $U$ always denotes a standard uniform distributed random variable. The central question of this paper is to derive lower and upper bounds on the expectation of the product of $d\geq2$ random variables under dependence uncertainty. Specifically, we consider the problems
\begin{align}
	m&=\inf \limits_{\forall i\ X_i\sim F_i}\mathbb{E}\left(X_1X_2\cdots X_d\right); \label{m} \\
	M&=\sup \limits_{\forall i\ X_i\sim F_i}\mathbb{E}\left(X_1X_2\cdots X_d\right).\label{M}
\end{align}
When $d=2$, it is well-known that $M$ is given by  $\mathbb{E}(F_1^{-1}(U)F_2^{-1}(U))$ (comonotonicity), and $m$ is given by $\mathbb{E}(F_1^{-1}(U)F_2^{-1}(1-U))$ (antimonotonicity). 
 For general $d$, 
the lower bound problem has a long history when $X_i \sim U[0,1]$,  \citep[see e.g.,][]{ruschendorf1980inequalities, bertino1994minimum, nelsen2012directional}. Specifically, \cite{wang2011complete} found the following closed-form expression (Corollary 4.1) for $m$ in the case of $X_i\sim U[0,1]$:
\begin{equation}\label{min of unif mix}
		m=\frac{1}{(d-1)^2}\left(\frac{1}{d+1}-(1-(d-1)c_d)^d+\frac{d}{d+1}(1-(d-1)c_d)^{d+1}\right)+(1-dc_d)c_d(1-(d-1)c_d)^{d-1},
\end{equation} 
where $c_d$ is the unique solution to $\log\left(1-d+\frac{1}{c}\right)=d-d^2c.$ Clearly, for $M$ we find in this case that $M=\mathbb{E}(U^d)=\frac{1}{d+1}$. When $X_i\sim U[a,b]$ such that $0<a<b<\infty$ and the inequality $\exp\{\frac{d}{b-a}(b(\ln b-1)-a(\ln a -1))\}-ab^{d-1}\leq 0$ holds, \cite{bignozzi2015studying} found the analytic result $m=\left(\frac{b^be^a}{a^ae^b}\right)^{\frac{d}{b-a}}$.

However, as far as we know, there are no other results available in the literature for computing $m$ and $M$ in more general cases. In the following section, we contribute to the literature by solving explicitly Problems \eqref{m} and \eqref{M} under various  assumptions on the marginal distributions $F_i$, $i=1,2,\dots,d$ (Section 3.1), or via an algorithm for arbitrary choices of $F_i$ (Section 3.2). 

\section{Lower and upper bounds} \label{main theorem}

\subsection{Analytic results}
We first provide lower and upper bounds when the $F_i$ are symmetric and have zero means. Next we study the case in which the $F_i$ satisfy some domain constraints or when they are uniform distributions on $(a,b)$, $a<0<b$.  
Our results make use of the following two lemmas. 
\begin{lemma}[Maximum product] \label{max product}
	Let $X_i\sim F_i$, denote by $G_i$ the df of the absolute value of $X_i$ $($i.e., $\lvert X_i\rvert\sim G_i$$)$, $i=1,2,\dots,d,$ and $U\sim U[0,1]$. Then,
 \begin{equation}\label{max ineq}
			M\leq\mathbb{E}\left(\prod_{i=1}^{d}G_i^{-1}(U)\right).
		\end{equation}
 If $\lvert X_1\rvert$, $\lvert X_2\rvert$, $\dots$, $\lvert X_d\rvert$ are comonotonic and $\prod_{i=1}^{d}X_i\geq 0$ a.s., then $(X_1, X_2,\dots, X_d)$ attains the maximum value $M$ and equality holds in \eqref{max ineq}.
\end{lemma}
\begin{proof}
As for the first part of the lemma, note that for any random vector $(Y_1, Y_2,\dots, Y_d)$ such that $Y_i\sim F_i,$ $i=1,2,\dots,d,$ it holds that $\mathbb{E}(\prod_{i=1}^{d}Y_i)\leq \mathbb{E}(\prod_{i=1}^{d}\lvert Y_i\rvert).$  \eqref{max ineq} then follows from the well-known fact that the right hand side of this inequality is maximized under a comonotonic dependence among the $\abs{Y_i}$. As for the second part, since $\lvert X_i\rvert\overset{d}{=}\lvert Y_i\rvert,$ the  $\abs{X_i}$ are comonotonic, and $\prod_{i=1}^{d}X_i\geq 0,$ it follows that $\mathbb{E}(\prod_{i=1}^{d}Y_i)\leq \mathbb{E}(\prod_{i=1}^{d}\lvert X_i\rvert)=\mathbb{E}(\prod_{i=1}^{d}X_i).$ 
\end{proof}
\begin{lemma}[Minimum product] \label{min product}
	Let $X_i\sim F_i$, denote by $G_i$ the df of the absolute value of $X_i$ $($i.e., $\lvert X_i\rvert\sim G_i$$)$, $i=1,2,\dots,d,$ and $U\sim U[0,1]$. Then, \begin{equation}\label{min ineq}
			m\geq-\mathbb{E}\left(\prod_{i=1}^{d}G_i^{-1}(U)\right).
		\end{equation} 
 If $\lvert X_1\rvert$, $\lvert X_2\rvert$, $\dots$, $\lvert X_d\rvert$ are comonotonic and $\prod_{i=1}^{d}X_i\leq 0$ a.s., then $(X_1, X_2,\dots, X_d)$ attains the minimum value $m$ and equality holds in \eqref{min ineq}.
\end{lemma}
\begin{proof}
	The proof is similar to the proof of the previous lemma noting that $\mathbb{E}(\prod_{i=1}^{d}X_i)=-\mathbb{E}(\prod_{i=1}^{d}\lvert X_i\rvert)$ when $\prod_{i=1}^{d}X_i\leq 0$ a.s..
\end{proof}

\subsubsection{Symmetric marginal distributions}
The following two theorems are main contributions of this paper.
\begin{theorem}[Upper bound]\label{upper bound}
	Let $F_i$, $i=1,2,\dots,d$, be symmetric with zero means and $U\sim U[0,1]$. There exists a random vector $(X_1,X_2,\dots, X_d)$ such that the $\lvert X_1\rvert$, $\lvert X_2\rvert$, $\dots$, $\lvert X_d\rvert$ are comonotonic and $\prod_{i=1}^{d}X_i\geq 0$ a.s.. Hence, $M=\mathbb{E}(\prod_{i=1}^{d}G_i^{-1}(U))$ where $G_i$ denotes the df of $\lvert X_i\rvert$. Furthermore, if $d$ is odd, then $X_i=F_i^{-1}(U_i)$ with
	\begin{equation}\label{max copula for odd dimension}
		\begin{aligned}
			U_1&=U_2=\cdots=U_{d-2}=U, \\
			U_{d-1}&=IJU+I(1-J)(1-U)+(1-I)JU+(1-I)(1-J)(1-U), \\
			U_d&=IJU+I(1-J)(1-U)+(1-I)J(1-U)+(1-I)(1-J)U,
		\end{aligned}
	\end{equation}
	where $I=\mathds{1}_{U>\frac{1}{2}}$, $J=\mathds{1}_{V>\frac{1}{2}}$, and $V\overset{d}{=}U[0, 1]$ is independent of $U$. If $d$ is even, then $X_i=F_i^{-1}(U_i)$ with \begin{equation}\label{max copula for even dimension}
		U_1=U_2=\cdots=U_d=U.
	\end{equation} 
\end{theorem}
\begin{proof}
	(1) $d$ is odd. The random variables $X_j$, $j=1,2,\dots,d-2$, can be expressed as follows:
	\begin{equation*}
		X_j=IJF^{-1}_j(U)+I(1-J)F^{-1}_j(U)+(1-I)JF^{-1}_j(U)+(1-I)(1-J)F^{-1}_j(U).
	\end{equation*}
	Furthermore,
	\begin{equation*}
		\begin{aligned}
			X_{d-1}&=IJF^{-1}_{d-1}(U)+I(1-J)F^{-1}_{d-1}(1-U)+(1-I)JF^{-1}_{d-1}(U)+(1-I)(1-J)F^{-1}_{d-1}(1-U), \\
			X_d&=IJF^{-1}_d(U)+I(1-J)F^{-1}_d(1-U)+(1-I)JF^{-1}_d(1-U)+(1-I)(1-J)F^{-1}_d(U).
		\end{aligned}
	\end{equation*}
	It follows that
	\begin{equation*}
		\begin{aligned}
			\abs{X_j}&=IJF^{-1}_j(U)+I(1-J)F^{-1}_j(U)-(1-I)JF^{-1}_j(U)-(1-I)(1-J)F^{-1}_j(U) \\
			&=IF^{-1}_j(U)-(1-I)F^{-1}_j(U), \\
			\abs{X_{d-1}}&=IJF^{-1}_{d-1}(U)-I(1-J)F^{-1}_{d-1}(1-U)-(1-I)JF^{-1}_{d-1}(U)+(1-I)(1-J)F^{-1}_{d-1}(1-U) \\
			&=IJF^{-1}_{d-1}(U)+I(1-J)F^{-1}_{d-1}(U)-(1-I)JF^{-1}_{d-1}(U)-(1-I)(1-J)F^{-1}_{d-1}(U) \\
			&=IF^{-1}_{d-1}(U)-(1-I)F^{-1}_{d-1}(U), \\
			\abs{X_d}&=IJF^{-1}_d(U)-I(1-J)F^{-1}_d(1-U)+(1-I)JF^{-1}_d(1-U)-(1-I)(1-J)F^{-1}_d(U) \\
			&=IJF^{-1}_d(U)+I(1-J)F^{-1}_d(U)-(1-I)JF^{-1}_d(U)-(1-I)(1-J)F^{-1}_d(U) \\
			&=IF^{-1}_{d}(U)-(1-I)F^{-1}_{d}(U), \\
		\end{aligned}
	\end{equation*}
	where we used in the second equations for $\abs{X_{d-1}}$ and $\abs{X_d}$ that $F_{d-1}$ resp. $F_{d}$ is symmetric. Note that the $\lvert X_i\rvert$ also write as $\lvert X_i\rvert=F_i^{-1}(Z)$, $i=1,2,\dots, d$, where $Z=U$ if $U\geq\frac{1}{2}$ and $Z=1-U$ if $U<\frac{1}{2}$,  
	i.e., $Z=\frac{1}{2}+\lvert U-\frac{1}{2}\rvert$. Next, we show that $\lvert X_1\rvert, \lvert X_2\rvert, \dots, \lvert X_d\rvert$ are comonotonic and that $\prod_{i=1}^{d}X_i\geq 0$ a.s.. First, $\lvert X_i\rvert$, $i=1,2,\dots, d$, are comonotonic because they are all increasing functions of $\lvert U-\frac{1}{2}\rvert$. Second,
	\begin{equation*}
		\begin{aligned}
			\prod_{i=1}^{d}X_i=&IJ\prod_{i=1}^{d}F^{-1}_i(U)+I(1-J)F^{-1}_{d-1}(1-U)F^{-1}_{d}(1-U)\prod_{i=1}^{d-2}F^{-1}_i(U)+\\
			&(1-I)J\prod_{i=1}^{d-1}F^{-1}_i(U)F^{-1}_{d}(1-U)+(1-I)(1-J)F^{-1}_{d-1}(1-U)F^{-1}_d(U)\prod_{i=1}^{d-2}F^{-1}_1(U),
		\end{aligned}
	\end{equation*}
	which is greater than or equal to zero because $F^{-1}_i(U)>0$ and $F^{-1}_i(1-U)\leq0$ when $U>\frac{1}{2}$, and $F^{-1}_i(U)\leq 0$ and $F^{-1}_i(1-U)>0$ when $U\leq\frac{1}{2}$, and where we use that $d$ is odd. Therefore, the vector $(X_1,X_2,\dots, X_d)$ with $X_i=F_i^{-1}(U_i)$, in which the $U_i$ are given as in \eqref{max copula for odd dimension}, attains $M$. 
	
	(2) $d$ is even. The random variables $X_i$ can be expressed as  $X_i=IF^{-1}_i(U)+(1-I)F^{-1}_i(U)$. Hence, $\lvert X_i\rvert=IF^{-1}_i(U)-(1-I)F^{-1}_i(U)$. It is clear that $\lvert X_1\rvert, \lvert X_2\rvert, \dots, \lvert X_d\rvert$ are comonotonic, as they are all increasing in $\lvert U-\frac{1}{2}\rvert$. Furthermore, $\prod_{i=1}^{d}X_i=I\prod_{i=1}^{d}F^{-1}_i(U)+(1-I)\prod_{i=1}^{d}F^{-1}_i(U),$
	which is greater than or equal to zero (note that $d$ is even). Therefore, the random vector $(X_1,X_2,\dots, X_d)$ with $X_i=F_i^{-1}(U)$ attains $M$. 
\end{proof}
\begin{theorem}[Lower bound]\label{lower bound}
	Let $F_i$, $i=1,2,\dots,d$, be symmetric with zero means and $U\sim U[0,1]$. There exists a random vector $(X_1,X_2,\dots, X_d)$ such that $\lvert X_1\rvert$, $\lvert X_2\rvert$, $\dots$, $\lvert X_d\rvert$ are comonotonic and $\prod_{i=1}^{d}X_i\leq 0$ a.s.. Hence, $m=-\mathbb{E}(\prod_{i=1}^{d}G_i^{-1}(U))$ where $G_i$ is the df of $\lvert X_i\rvert$. Furthermore, if $d$ is odd, then $X_i=F_i^{-1}(U_i)$ with
	\begin{equation}\label{min copula for odd dimension}
		\begin{aligned}
			U_1&=U_2=\cdots=U_{d-2}=U, \\
			U_{d-1}&=IJU+I(1-J)(1-U)+(1-I)JU+(1-I)(1-J)(1-U), \\
			U_d&=IJ(1-U)+I(1-J)U+(1-I)JU+(1-I)(1-J)(1-U),
		\end{aligned}
	\end{equation}
	where $I=\mathds{1}_{U>\frac{1}{2}}$, $J=\mathds{1}_{V>\frac{1}{2}}$, and $V\overset{d}{=}U[0, 1]$ is independent of $U$. If $d$ is even, then $X_i=F_i^{-1}(U_i)$ with \begin{equation}\label{min copula for even dimension}
		\begin{aligned}
			U_1&=U_2=\cdots=U_{d-1}=U\quad \hbox{and} \quad U_d=1-U.
		\end{aligned}
	\end{equation}
\end{theorem}
We omit the proof of the lower bound because it is similar to that of the upper bound. Figure~\ref{optimal copulas} presents the supports of the copulas in \eqref{max copula for odd dimension} and \eqref{min copula for odd dimension} when $d=3$. As the projections of the support on the planes formed by the x-axis and y-axis, resp.\ x-axis and z-axis, resp.\ y-axis and z-axis form crosses, we label these copulas as \textit{cross product} copulas. Note that the simulation shows the densities of the cross product copulas are uniform on each of the segments.

\begin{figure}[!h]
	\centering
	\includegraphics[width=0.9\textwidth]{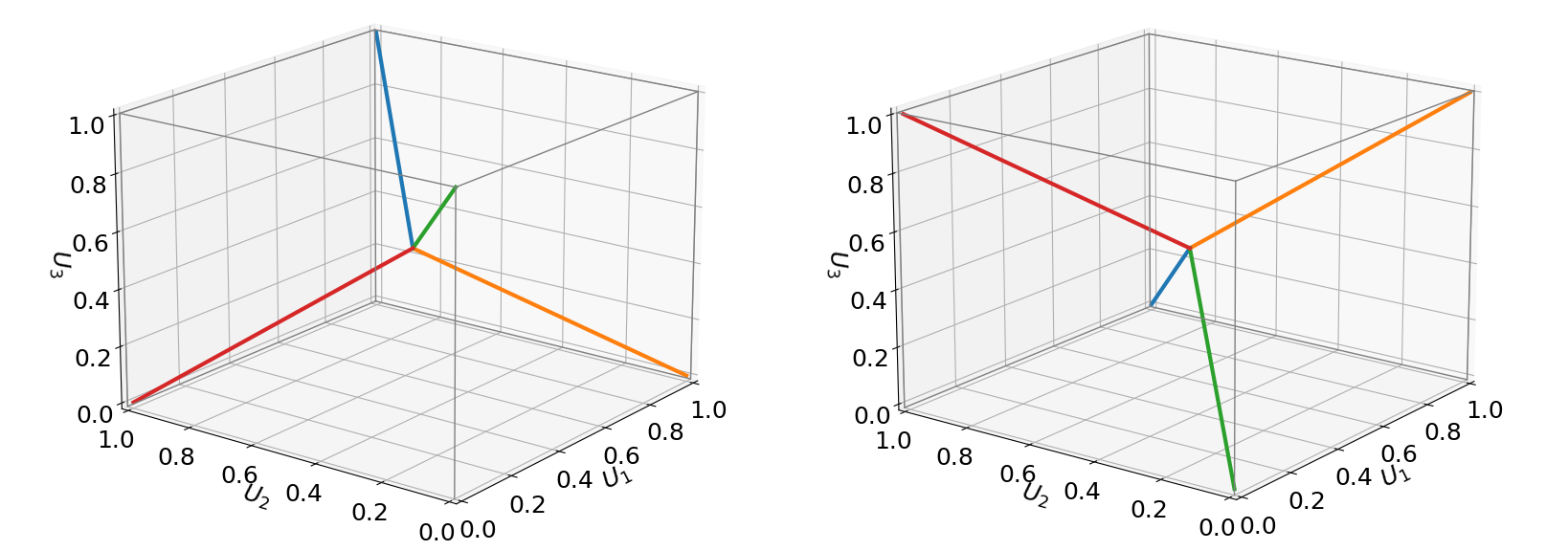}
	\caption{Support of the (cross product) copula that maximizes (left panel) resp.\ minimizes (right panel) $\mathbb{E}(X_1X_2X_3)$ where $X_i\sim F_i$ ($i=1,2,3$) in the case that the $F_i$ are symmetric with zero means.}
	\label{optimal copulas}
\end{figure}
\begin{corollary}\label{uniform margins in d dimensions}
	Let the $F_i$ be uniform distributions on $\left[-\sqrt{3},\sqrt{3}\right]$. Then $M=\left(\sqrt{3}\right)^{d}/(d+1)$ and is attained by a random vector $(X_1,X_2,\dots, X_d)$, where $X_i=F_i^{-1}(U_i)$ in which the $U_i$ are given in \eqref{max copula for odd dimension} $($resp., \eqref{max copula for even dimension}$)$ if $d$ is odd $($resp., even$)$. Furthermore, $m=-\left(\sqrt{3}\right)^{d}/(d+1)$ and is attained by a random vector $(X_1,X_2,\dots, X_d)$ where $X_i=F_i^{-1}(U_i)$ in which the $U_i$ are given in \eqref{min copula for odd dimension} $($resp., \eqref{min copula for even dimension}$)$ if $d$ is odd $($resp., even$)$.
\end{corollary}
\begin{proof}
	With $X_i\sim \left[-\sqrt{3},\sqrt{3}\right]$, we find that $|X_i|:=V\sim U\left[0,\sqrt{3}\right]$. Then we find from Theorem \ref{upper bound} that $M = \mathbb{E}(V^{d})=\int_{0}^{\sqrt{3}}v^d\frac{\sqrt{3}}{3}dv=\frac{3^{\frac{d}{2}}}{d+1}.$
	In a similar way, we find from Theorem \ref{lower bound} that $m=-3^{\frac{d}{2}}/(d+1)$.
\end{proof}
\begin{remark}
	For general distribution functions $F_i$, $i=1,2,\dots,d$, the upper bound in \eqref{max ineq} and the lower bound in \eqref{min ineq} are typically not attainable. Moreover, the construction $X_i=F_i^{-1}(U_i)$ with the $U_i$ as given in \eqref{max copula for odd dimension} ($d$ is odd) resp.\ as given in \eqref{max copula for even dimension} ($d$ is even) does not lead to the sharp bound $M$ (similar for the case of the lower bound  $m$). To illustrate this point, let the $F_i$, $i=1,2,3,$ denote discrete distributions having mass points -1 and 10 with equal probability. Under a comonotonic dependence among the $X_i\sim F_i$, we obtain that $\mathbb{E}\left(X_1X_2X_3\right)=499.5$. However, under the dependence as in \eqref{max copula for odd dimension}, we obtain that $\mathbb{E}\left(X_1X_2X_3\right)=257.5 <499.5$. Moreover, $\mathbb{E}\left(\lvert X_1\rvert\lvert X_2\rvert\lvert X_3\rvert\right)=500.5$ is not attainable.
\end{remark}

\subsubsection{Marginal distributions under domain restrictions}

In this subsection, we provide sharp bounds under various conditions that the domain of $F_i$ is non-negative or non-positive, or that the $F_i$ are uniforms on $(a,b)$, $a<0<b$. 

\begin{proposition}[Non-negative domain]\label{non-neg}
	Let $X_i\sim F_i$ in which the $F_i$ have non-negative domain. 
	\begin{enumerate}[$(1)$]
		\item The upper bound $M$ is attained when $(X_1,X_2,\dots, X_d)$ is a comonotonic random vector, i.e., $X_i=F_i^{-1}(U)$ in which $U \sim U[0,1]$.
		\item Under a mixing assumption on the distributions of $\ln X_1, \ln X_2, \dots, \ln X_d$, i.e., $\sum_{i=1}^{d}\ln X_i=c$, where $c$ is a constant, it holds that $m=e^c$ and $m$ is attained by $(X_1,X_2,\dots,X_d)$.
	\end{enumerate}
\end{proposition}
\begin{proof}
	The first statement follows from Lemma \ref{max product} in a direct manner. As for the proof of the second statement, it holds for any $Y_i\sim F_i$, $i=1,2,\dots,d$ that $\mathbb{E}\left(\prod_{i=1}^{d}Y_i\right)=\mathbb{E}\left(\exp \left(\sum_{i=1}^{d}\ln Y_i\right)\right)\geq \exp\left(\mathbb{E}\left(\sum_{i=1}^{d}\ln Y_i\right)\right)$ where the last inequality follows from Jensen's inequality. Furthermore, this inequality turns into an equality when  $\sum_{i=1}^{d}\ln Y_i $ is constant. This implies the second statement.
\end{proof}
\begin{remark}
	For k-th order mixed moments and $X_i \sim F_i$ in which the $F_i$ have non-negative domains, $i=1,2,\dots,d$, we obtain from Proposition \ref{non-neg} the lower and upper bound of $\mathbb{E}\left(X_1^{k_1}X_2^{k_2}\cdots X_d^{k_d}\right)$. Under a mixing assumption on the distributions of $k_1\ln X_1,k_2\ln X_2,\dots,k_d\ln X_d$, it holds that the lower bound is attained by $(X_1,X_2,\dots,X_d)$. The mixing conditions for this case are given in Section 3 of \cite{wang2016joint}. In particular, this holds true in case $k_1 = k_2 =\cdots=k_d=l$ in which $l$ is an integer and $l\geq 1$, and the lower bound is $m^{l}$, where $m$ is the value in \eqref{min of unif mix}.
\end{remark}
\begin{proposition}[Non-positive domain]
	\label{non_pos}
	Let $X_i\sim F_i$ in which the $F_i$ have non-positive domain and $U \sim U[0,1]$. 
	\begin{enumerate}[$(1)$]
		\item Let $d$ be an odd number. Under a mixing assumption on the distributions of $\ln \lvert X_1\rvert, \ln\lvert X_2\rvert, \dots,$ $\ln\lvert X_d\rvert$, i.e., $\sum_{i=1}^{d}\ln\lvert X_i\rvert=c$, where $c$ is a constant, it holds that $M=-e^c$ and $M$ is attained by $(X_1,X_2,\dots,X_d)$. $m$ is attained when $(X_1,X_2,\dots, X_d)$ is a comonotonic random vector, i.e., $X_i=F_i^{-1}(U)$.
		\item Let $d$ be an even number. $M$ is attained when $(X_1,X_2,\dots,X_d)$ is a comonotonic random vector, i.e., $X_i=F_i^{-1}(U)$. Under a mixing assumption on the distributions of $\ln\lvert X_1\rvert, \ln\lvert X_2\rvert, \dots, \ln\lvert X_d\rvert$, i.e., $\sum_{i=1}^{d}\ln\lvert X_i\rvert=c$, where $c$ is a constant, it holds that $m=e^c$ and $m$ is attained by $(X_1,X_2,\dots,X_d)$.
	\end{enumerate}
\end{proposition}
\proof Its proof is similar to that of Proposition \ref{non-neg}, we thus omit it.\hfill$\Box$

Proposition \ref{non_pos} shows that when the $F_i$ have non-positive domain and $d$ is odd, an upper bound on $M$ is given by $-\exp\left(\mathbb{E}\left(\sum_{i=1}^{d}\ln\lvert Y_i\rvert\right)\right)$, $Y_i \sim F_i$.  \cite{wang2011complete, wang2015extreme}, \cite{puccetti2015extremal}, and  \cite{puccetti2012advances} provide general conditions on the $F_i$ that ensure the construction of $X_i\sim F_i$ such that the distributions of $X_1,X_2, \dots, X_d$ are mixing and thus allow to infer sharpness of the bounds above. For early results of this type, see \cite{gaffke1981class} and \cite{ruschendorf2002n}. 

\begin{proposition}[Uniform distributions with non-zero means]
	\label{prop31}
	Let $X_i\sim F_i$, $i=1,2,\dots, d$. Assume that $d$ is odd and that the $F_i$ are uniform distributions on $[a, b]$ $(a<0<b).$  Define $J=\mathds{1}_{V>\frac{1}{2}}$, in which $V\sim U[0, 1]$ is independent of $U\sim U[0, 1]$. It holds that:
	\begin{enumerate}[$(1)$]
		\item Let $\abs{a}<b$ and $c=\frac{-2a}{b-a}$. $M$ is attained by a random vector $(X_1,X_2,\dots, X_d),$ with $X_i=F_i^{-1}(U_i)$ in which
		\begin{align}
				U_1=&U_2=\cdots=U_{d-2}=U, \label{strong positive}\\
				U_{d-1}=&(1-I)[KJU+K(1-J)(c-U)+(1-K)JU+(1-K)(1-J)(c-U)]+IU, \notag\\
				U_d=&(1-I)[KJU+K(1-J)(c-U)+(1-K)J(c-U)+(1-K)(1-J)U]+IU,\notag
		\end{align}
		and where $I=\mathds{1}_{U> c}$ and $K=\mathds{1}_{U>\frac{c}{2}}$. 
		\item Let $\abs{a}>b$ and $c=\frac{-b-a}{b-a}$. 
		$m$ is attained by the random vector $(X_1,X_2,\dots, X_d)$, with  $X_i=F_i^{-1}(U_i)$ in which 
			\begin{align}
				U_1=&U_2=\cdots=U_{d-2}=U, \label{strong negative}\\
				U_{d-1}=&(1-I)[KJU+K(1-J)(1+c-U)+(1-K)JU+(1-K)(1-J)(1+c-U)]+IU, \notag\\
				U_d=&(1-I)[KJ(1+c-U)+K(1-J)U+(1-K)JU+(1-K)(1-J)(1+c-U)]+IU,\notag
			\end{align}
			
		and where $I=\mathds{1}_{U<c}$ and $K=\mathds{1}_{U>\frac{1+c}{2}}$. 
	\end{enumerate} 
\end{proposition}
\proof	\begin{enumerate}[$(1)$]
		\item Note that $F_i^{-1}(u)>0$ if and only if $u>\frac{c}{2}$. When $1\leq j\leq d-2,$ $X_j=F^{-1}_{j}(U).$ Moreover,
		\begin{equation*}
			\begin{aligned}
				X_{d-1}=&(1-I)[KJF^{-1}_{d-1}(U)+K(1-J)F^{-1}_{d-1}(c-U)+(1-K)JF^{-1}_{d-1}(U)\\
				&+(1-K)(1-J)F^{-1}_{d-1}(c-U)]+IF^{-1}_{d-1}(U), \\
				X_d=&(1-I)[KJF^{-1}_d(U)+K(1-J)F^{-1}_d(c-U)+(1-K)JF^{-1}_d(c-U)\\
				&+(1-K)(1-J)F^{-1}_d(U)]+IF^{-1}_d(U).
			\end{aligned}
		\end{equation*}
	The absolute values of $X_i$ are
	\begin{equation*}
		\begin{aligned}
			\abs{X_j}=&(1-I)[KF^{-1}_{j}(U)-(1-K)F^{-1}_{j}(U)]+IF^{-1}_{j}(U) \\
							\end{aligned}
	\end{equation*}	\begin{equation*}
		\begin{aligned}\abs{X_{d-1}}=&(1-I)[KJF^{-1}_{d-1}(U)-K(1-J)F^{-1}_{d-1}(c-U)-(1-K)JF^{-1}_{d-1}(U)\\
			&+(1-K)(1-J)F^{-1}_{d-1}(c-U)]+IF^{-1}_{d-1}(U) \\
			=&(1-I)[KF^{-1}_{d-1}(U)-(1-K)F^{-1}_{d-1}(U)]+IF^{-1}_{d-1}(U), \\
			\abs{X_d}=&(1-I)[KJF^{-1}_d(U)-K(1-J)F^{-1}_d(c-U)+(1-K)JF^{-1}_d(c-U)\\
			&-(1-K)(1-J)F^{-1}_d(U)]+IF^{-1}_d(U) \\
			=&(1-I)[KF^{-1}_{d}(U)-(1-K)F^{-1}_{d}(U)]+IF^{-1}_{d}(U). \\
		\end{aligned}
	\end{equation*}
The above equations for $\abs{X_{d-1}}$ and $\abs{X_d}$ hold because $-F^{-1}_d(c-U)=F^{-1}_d(U)$ if $U \leq c$. Similarly to Theorem \ref{upper bound}, we apply Lemma \ref{max product} to prove this proposition. First, $\lvert X_i\rvert$, $i=1,2,\dots, d$, are comonotonic because they are all increasing functions of $\lvert U-\frac{c}{2}\rvert$ ($\lvert X_i\rvert = F^{-1}_{i}(Z)$ where $Z=\frac{c}{2}+\lvert U-\frac{c}{2}\rvert$). Moreover, it verifies that $\prod_{i=1}^{d}X_i\geq 0$ if $d$ is odd. Hence, $M$ is attained by the random vector $(X_1,X_2,\dots, X_d)$, where $X_i=F_i^{-1}(U_i)$ with $U_i$ in \eqref{strong positive}. 
		\item The proof of (2) is similar to that of (1) and thus omitted.\hfill$\Box$
	\end{enumerate} 

Figure~\ref{optimal copulas unif} displays the supports of the copulas in \eqref{strong positive} and \eqref{strong negative} when $d=3$ and $a<0$.

\begin{figure}[H]
	\centering
	\includegraphics[width=0.95\textwidth]{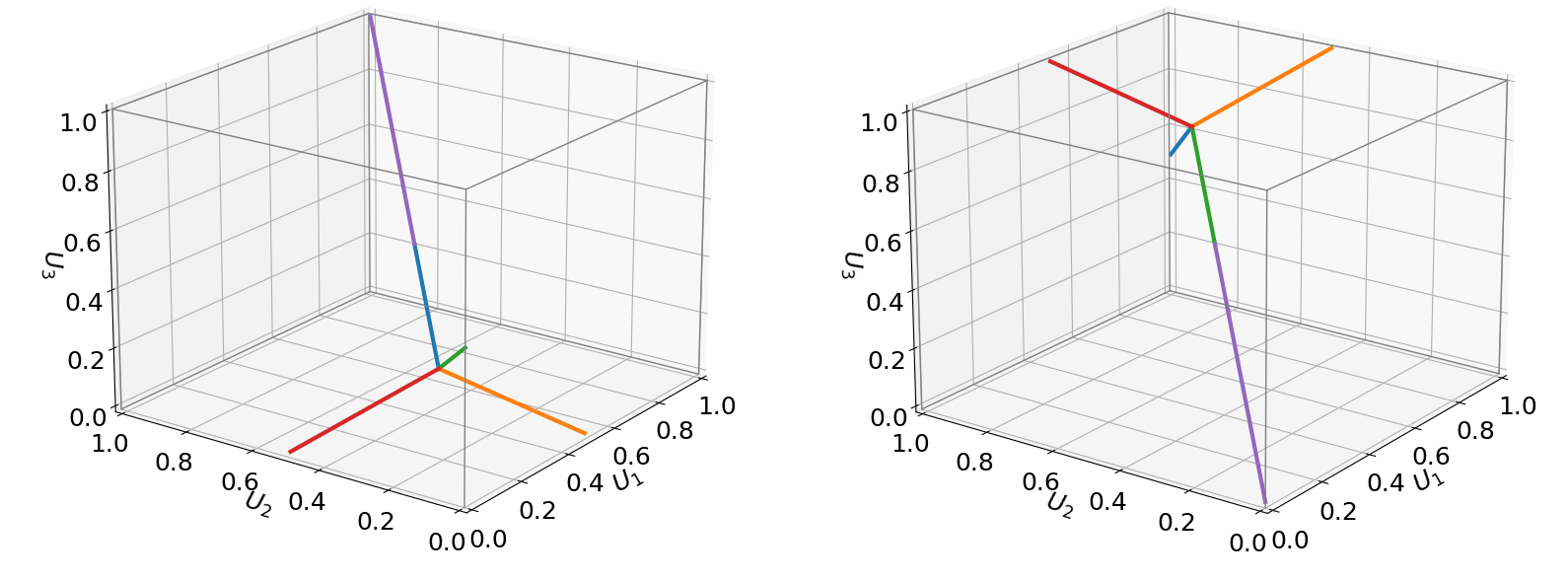}
	\caption{Support of the copula \eqref{strong positive} (resp., \eqref{strong negative}) that maximizes (left panel) (resp., minimizes (right panel)) $\mathbb{E}(X_1X_2X_3)$ with $F_i\sim U[a,-3a]$ (resp., $F_i\sim U[a,-a/3]$), in which $a<0$ and $c=\frac{1}{2}$.
	\label{optimal copulas unif}	}
\end{figure}

%


\subsection{Algorithm for obtaining sharp bounds}\label{Simulated Annealing Algorithm}
In this subsection, we develop an algorithm to approximate for any given choice of $F_i,$ $i=1,2,\dots, d$, the sharp bounds $m$ and $M$. The algorithm is based on the following lemma that establishes necessary conditions that the solutions to the optimization problems \eqref{M} resp.\ \eqref{m} need to satisfy. 
\begin{lemma}\label{basis of bra}
	If $(X_1,X_2,\dots, X_d)$ solves problem \eqref{M} $($resp., \eqref{m}$)$, then for any choice of subsets $I$ of $\{1, 2, \dots, d\}$, it holds that $X_1=\prod_{i\in I}X_i$ and $X_2=\prod_{i\notin I}X_i$ are comonotonic $($resp., antimonotonic$)$.
\end{lemma}
Making use of Lemma~\ref{basis of bra}, we can now design an algorithm to obtain approximate solutions to problems \eqref{m} and \eqref{M}.

\begin{algorithm} \label{BRA algorithm}
	\begin{enumerate}
		\item[]
		\item Simulate $n$ draws $u_{j}$, $j=1,2,\dots, n$, from a standard uniform distributed random variable.
		\item Initialize $n\times d$ matrix $\bm{X}=(x_1, x_2,\dots, x_d)$ where $x_i=(x_{1i}, x_{2i},\dots, x_{ni})^T$ denotes the $i$-th column ($i=1,2,\dots,d$) and $x_{ji}=F_i^{-1}(u_j)$.
		\item \label{rearrangement} Rearrange two blocks of the matrix $\bm{X}$:
		\begin{enumerate}[3.1.]
			\item Select randomly a subset $I$ of $\{1, 2, \dots, d\}$ of cardinality lower than or equal to $\frac{d}{2}$.
			\item Separate two blocks (submatrices) $\bm{X_1}$ and $\bm{X_2}$ from $\bm{X}$ where the first block $\bm{X_1}$ contains columns of $\bm{X}$ having index in $I$ and the second block $\bm{X_2}$ consists of the other columns.
			\item Rearrange (swap) the rows of first block so that the vector $x_1=(\prod\limits_{i\in I}x_{1i}, \prod\limits_{i\in I}x_{2i},\dots, \prod\limits_{i\in I}x_{ni})^T$ is comonotonic (resp., antimonotonic) to $x_2=(\prod\limits_{i\notin I}x_{1i}, \prod\limits_{i\notin I}x_{2i},\dots, \prod\limits_{i\notin I}x_{ni})^T$ in the case of problem \eqref{M} (resp., \eqref{m}). 
			\item Compute $\Lambda=\frac{1}{n}\sum\limits_{j=1}^{n}\left(\prod\limits_{i=1}^{d}x_{ji}\right)$.
		\end{enumerate}
		\item If there is no difference\footnote{ On the one hand, if the dimension $d$ is large and the algorithm converges slowly, the stop criteria we use is that the relative change in the value of $\Lambda$ is less than 0.01\%. On the other hand, for small dimensions (typically when $d$ is less than 30), it is possible to perform steps 3.2, 3.3 and 3.4 for all possible subsets instead of only 50 randomly chosen subsets. The necessary condition from Lemma \ref{basis of bra} is then guaranteed to be satisfied. } in $\Lambda$ after 50 steps of Step \ref{rearrangement}, output the current matrix $X$ and $\Lambda$, otherwise return to step \ref{rearrangement}.
	\end{enumerate}
\end{algorithm}

To illustrate the empirical performance of the algorithm, we compare in case $F_i\sim U[0,1]$ $(i=1,2,\dots,d)$ the analytic result of \cite{wang2011complete} for the lower bound $m$ with the numerical value obtained by applying the algorithm. In Table~\ref{compare risk bounds}, we report the cases $d=3,  5, 10,  50$ and $n=1000,10000,100000$. We observe that the approximate value is not significantly different   

\begin{table}[H]
	\centering 
	\def\arraystretch{1.8}
	\begin{tabular}{|c|c|c|c|c|}
		\hline
		d&Analytic value & $n=1000$ & $n=10000$ & $n=100000$ \\
		\hline
		3&$5.4803\times 10^{-2}$ & \makecell{$5.4869\times 10^{-2}$ \\ ($1.6\times 10^{-4}$, 0.01s)}& \makecell{$5.4869\times 10^{-2}$ \\ ($5.0\times 10^{-6}$, 0.06s)} &\makecell{$5.4796\times 10^{-2}$ \\ ($1.6\times 10^{-6}$, 0.75s)}\\
		\hline
		5&$6.8604\times 10^{-3}$&\makecell{$6.9259\times 10^{-3}$ \\ ($3.5\times 10^{-5}$, 0.01s)}&\makecell{$6.8844\times 10^{-3}$ \\ ($1.1\times 10^{-5}$, 0.08s)}&\makecell{$6.8616\times 10^{-3}$ \\ ($3.3\times 10^{-6}$, 1.13s)}\\
		\hline
		10&$4.5410\times 10^{-5}$&\makecell{$4.8185\times 10^{-5}$ \\ ($4.9\times 10^{-7}$, 0.01s)}&\makecell{$4.5924   \times 10^{-5}$ \\ ($1.4\times 10^{-7}$, 0.15s)} & \makecell{$4.5372\times 10^{-5}$ \\ ($4.4\times 10^{-8}$, 1.89s)}\\
		\hline
		50&$1.9287\times 10^{-22}$&\makecell{$6.2708\times 10^{-22}$ \\ ($4.5\times 10^{-23}$, 0.02s)}&\makecell{$2.2119\times 10^{-22}$ \\ ($3.7\times 10^{-24}$, 0.34s)} & \makecell{$1.9654\times 10^{-22}$ \\ ($9.56\times 10^{-25}$, 8.73s)}\\
		\hline
	\end{tabular}
	\caption{Let $F_i\sim U[0,1]$ for $i=1,2,\dots,d$. We compare the analytic value for $m$ from \cite{wang2011complete} with the numerical value obtained using Algorithm \ref{BRA algorithm} (mean across 1000 experiments) for $n=1000,10000,100000$. The numbers between parentheses represent the standard errors and average time consumption.}
\label{compare risk bounds}
\end{table}
\noindent from the analytic value (especially when $n$ is big). The run time increases if $d$ and $n$ increase. The standard errors illustrate that the algorithm we use is relatively stable. To summarize, our proposed  algorithm appears to be a simple, fast and stable method to numerically solve problems \eqref{m} and \eqref{M}.

\section{Application to coskewness uncertainty} \label{uniform margins}

In this section, we apply the results obtained so far to the study of risk bounds on coskewness among random variables $X_i$ with given marginal distributions $F_i$ ($i=1,2,3$) but unknown dependence.  
To begin with, the coskewness of $X_1$, $X_2$ and $X_3$, denoted by $S(X_1, X_2, X_3)$, is given as
	\begin{equation*} 
		S(X_1, X_2, X_3)=\frac{\mathbb{E}((X_1-\mu_{1})(X_2-\mu_{2})(X_3-\mu_{3}))}{\sigma_{1}\sigma_{2}\sigma_{3}},
		\label{equation for coskewness} 
	\end{equation*}
and we thus aim at solving the following problems
\begin{equation}\label{min cosk}
	\underline{S}=\inf \limits_{X_i\sim F_i,\ i=1,2,3}S(X_1,X_2,X_3),
\end{equation}
\begin{equation}\label{max cosk}
	\overline{S}=\sup \limits_{X_i\sim F_i,\ i=1,2,3}S(X_1,X_2,X_3).
\end{equation}
Note that $X_i\sim F_i\iff Y_i=\frac{X_i-\mu_i}{\sigma_i}\sim H_i$ where $H_i^{-1}=\frac{F_i^{-1}-\mu_i}{\sigma_i}$. Hence, solving Problems~\eqref{min cosk} and \eqref{max cosk} under the restriction $X_i\sim F_i$ ($i=1,2,3$) is equivalent to solving the optimizations problems \eqref{m} and \eqref{M} for the case $X_i\sim H_i$. That is, standardization of the marginal distributions $F_i$, $i=1,2,3$, does not affect the bounds.

\subsection{Risk bounds on coskewness}

The following proposition follows as a direct application of Theorem \ref{upper bound} resp.\ Theorem \ref{lower bound}. 

\begin{proposition} \label{symmetric margins and any copula}
	Let $X_i\sim F_i$ in which the $F_{i}$ are symmetric, $i=1,2,3$, and $U\sim U[0,1]$. The maximum coskewness $\overline{S}$ of $X_1$, $X_2$ and $X_3$ under dependence uncertainty is given as
	\begin{equation} \label{maximum coskewness}
		\overline{S}=\mathbb{E}\left(G^{-1}_{1}(U)G^{-1}_{2}(U)G^{-1}_{3}(U)\right)
	\end{equation} where $G_i$ is the df of $|(X_i-\mu_i)/\sigma_i|$ and is attained when $X_i=F_i^{-1}(U_i)$ with $U_i$ as in \eqref{max copula for odd dimension}; the minimum coskewness $\underline{S}$ is given as 
	\begin{equation} \label{minimum coskewness}
		\underline{S}=-\mathbb{E}\left(G^{-1}_{1}(U)G^{-1}_{2}(U)G^{-1}_{3}(U)\right)
	\end{equation} and is attained when $X_i=F_i^{-1}(U_i)$ with $U_i$ as in \eqref{min copula for odd dimension}. 
\end{proposition}



Thanks to Proposition~\ref{symmetric margins and any copula}, we can compute the risk bounds on coskewness for different choices of symmetric marginal distributions. 

\textbf{Uniform marginal distributions:} Let  $F_i\sim U[a_i,b_i]$, $i=1,2,3$. Standardization of the $F_i$ leads to marginal distributions $H_i \sim U\left[-\sqrt{3},\sqrt{3}\right]$. Hence, an application of Corollary~\ref{uniform margins in d dimensions} to the case $d=3$ yields that
 $\overline{S}=\frac{3\sqrt{3}}{4}$ and $\underline{S}=-\frac{3\sqrt{3}}{4}$.

\textbf{Normal marginal distributions:} Let  $F_i\sim N(\mu_i,\sigma^2_i)$, $i=1,2,3$. After standardization we find that 
$\overline{S}=\mathbb{E}\left(G^{-1}(U)^3\right)=2\mathbb{E}\left(Z^3\mathds{1}_{Z>0}\right)$ where $G$ is the df of $\abs{Z}$ with $Z\sim N(0, 1)$. Integration yields that
$\overline{S}=\frac{2}{\sqrt{2\pi}}\int_{0}^{+\infty}z^3e^{-\frac{z^2}{2}}dz=\frac{2\sqrt{2\pi}}{\pi}.
$

Similar calculations can also be performed for other symmetric marginal distributions. In Table \ref{general risk bounds}, we report risk bounds on coskewness according to Proposition~\ref{symmetric margins and any copula} for various cases. Note that except for the parameter $\nu$, all parameters in the table have no impact on the bounds because they are location and scale parameters.

\begin{table}[!htbp] 
	\centering 
	\renewcommand\arraystretch{1.5}
	\begin{tabular}{|c|c|c|}
		\hline
		Marginal Distributions $F_i$&Minimum Coskewness&Maximum Coskewness\\
		\hline
		$N(\mu_i, \sigma_i^2)$&$-\frac{2\sqrt{2\pi}}{\pi}$&$\frac{2\sqrt{2\pi}}{\pi}$\\
		\hline
		$Student(\nu),\ \nu>3$&$-\frac{4(\nu-2)\sqrt{(\nu-2)\pi}\Gamma(\frac{\nu+1}{2})}{(3-4\nu+\nu^2)\pi\Gamma(\frac{\nu}{2})}$&$\frac{4(\nu-2)\sqrt{(\nu-2)\pi}\Gamma(\frac{\nu+1}{2})}{(3-4\nu+\nu^2)\pi\Gamma(\frac{\nu}{2})}$\\
		\hline
		$Laplace(\mu_i, b_i)$&$-\frac{3\sqrt{2}}{2}$&$\frac{3\sqrt{2}}{2}$\\
		\hline
		$U[a_i, b_i]$&$-\frac{3\sqrt{3}}{4}$&$\frac{3\sqrt{3}}{4}$\\
		\hline
	\end{tabular}
	\caption{Maximum and minimum coskewness for various choices of the marginal distributions. $\Gamma(x)$ denotes the gamma function.}
\label{general risk bounds}

\end{table}

\begin{proposition}
	When $F_i$, $i=1,2,3$, are symmetric, then $\overline{S}$ and $\underline{S}$ 
	are opposite numbers.
\end{proposition}
\begin{proof}
We omit the proof since it is an immediate consequence of Theorems \ref{upper bound} and \ref{lower bound}.
\end{proof}

Based on these new bounds, we define hereafter a novel concept of standardized rank coskewness.

\subsection{Standardized rank coskewness}\label{sMsection standardized rank coskewness}
An important feature of the coskewness is that it depends on marginal distributions. 
In the same spirit as \cite{spearman1961proof} for the rank correlation, we propose to define the standardized rank coskewness among given variables $X_1\sim F_1$, $X_2 \sim F_2$ and $X_3 \sim F_3$ as the coskewness of the transformed variables $F_1(X_1)$, $F_2(X_2)$, and $F_3(X_3)$.
\begin{definition}[Standardized rank coskewness]\label{standardized rank coskewness}
	Let $X_i\sim F_{i}$, $i=1,2,3$, such that $F_i$ are strictly increasing and continuous. The standardized rank coskewness of $X_1$, $X_2$ and $X_3$ denoted by $RS(X_1,X_2,X_3)$ is defined as $RS(X_1,X_2,X_3)=\frac{4\sqrt{3}}{9}S(F_{1}(X_1), F_{2}(X_2), F_{3}(X_3))$. Hence,
	\begin{equation} 
		RS(X_1,X_2,X_3)=32\mathbb{E}\left(\left(F_{1}(X_1)-\frac{1}{2}\right)\left(F_{2}(X_2)-\frac{1}{2}\right)\left(F_{3}(X_3)-\frac{1}{2}\right)\right).
		\label{equation for rank coskewness} 
	\end{equation}
\end{definition}
\begin{proposition}\label{src proposition}
	Let $X_i\sim F_i$ for $i=1,2,3$. The standardized rank coskewness $RS(X_1,X_2,X_3)$ satisfies the following properties:
	\begin{enumerate}[(1)]
		\item $-1\leq RS(X_1,X_2,X_3)\leq 1$.
		\item The upper bound of $1$ 
		is obtained when the $X_i$ are of the form $X_i=F_i^{-1}(U_i)$ in which the $U_i$ are given in \eqref{max copula for odd dimension}. The lower bound of $-1$ is obtained when the $X_i$ are of the form $X_i=F_i^{-1}(U_i)$ in which the $U_i$ are given in \eqref{min copula for odd dimension}.
		\item It is invariant under strictly increasing transformations, i.e.,  when $f_i$, $i=1,2,3$, are arbitrary strictly increasing functions, we have $RS(X_1,X_2,X_3)=RS(f_1(X_1),f_2(X_2),f_3(X_3)).$
		\item $RS(X_1,X_2,X_3)=0$ if $X_1$, $X_2$ and $X_3$ are independent.
	\end{enumerate}
\end{proposition}
Note that $X_i \sim F_i$ ($i=1,2,3$) exhibit maximum resp.\ minimum standardized rank coskewness when they have a cross product copula specified through \eqref{max copula for odd dimension} resp.\ \eqref{min copula for odd dimension}. Specifically, the properties in (1)-(4) are a strong motivation for the introduction of the newly introduced notion of standardized rank coskewness. One shortcoming of the new definition like the traditional coskewness is that the last property in Proposition~\ref{src proposition} is sufficient but not necessary. 

\subsection{Asymmetric marginals}\label{application to broad case}
When the $F_i$ are not symmetric, one can still obtain explicit bounds on coskewness providing the $F_i$ satisfy some domain conditions; see Propositions \ref{non-neg}-\ref{prop31}. In the general case, one can invoke Algorithm~\ref{BRA algorithm} to obtain approximations for the sharp bounds. We examine hereafter the example of lognormal distributions, i.e., $F_i\sim\log N(0,1)$ for $i=1,2,3$. 
\begin{figure}[H]
	\centering
	\includegraphics[width=0.9\textwidth]{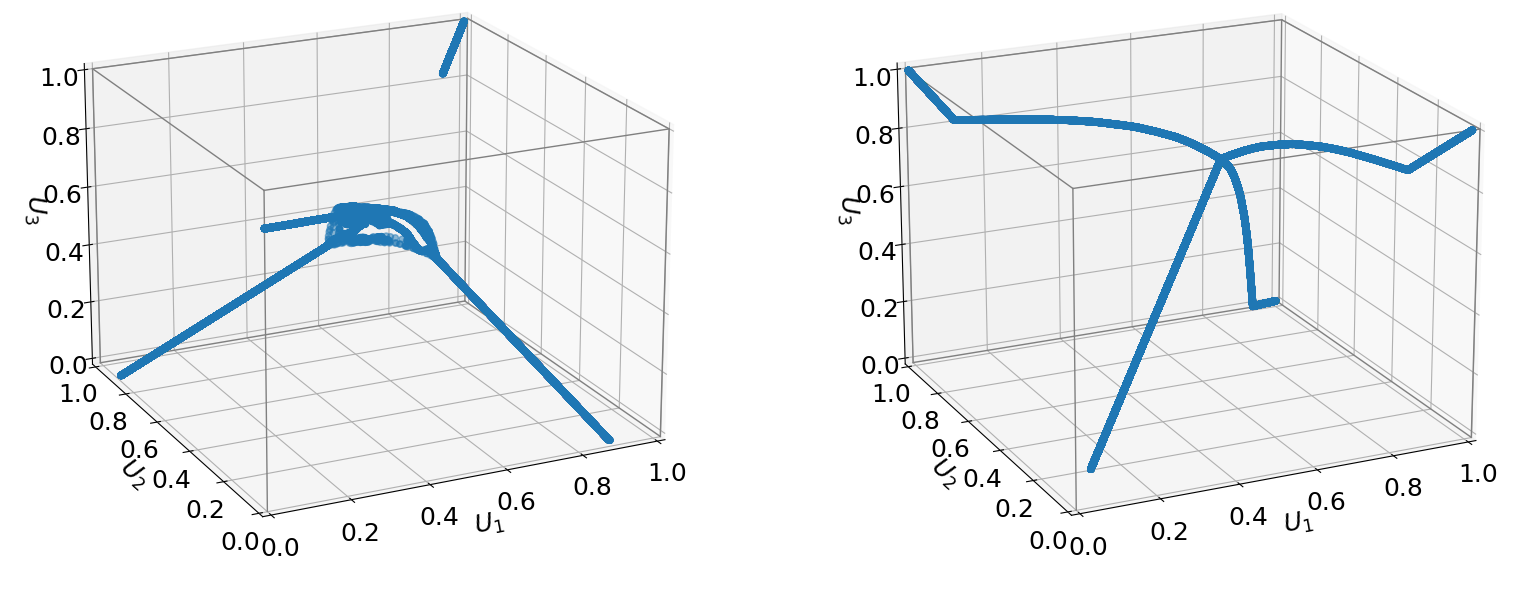}
	\caption{With $F_i\sim\log N(0,1)$ ($i=1,2,3$) and $n=100000$, the support of the copula that maximizes (resp.\ minimizes) coskewness is displayed in the left (resp.\ the right) panel. In this case, $M\approx 5.71$ and $m\approx -0.97$.}
	\label{lognorm margins}
\end{figure}

From Algorithm~\ref{BRA algorithm}, we obtain that maximum and minimum coskewness are approximately equal to $5.71$ resp.\ $-0.97$ when $n=100000$. The supports of the corresponding copulas are displayed in Figure~\ref{lognorm margins}. Note that using the copulas coming from \eqref{max copula for odd dimension} resp.\ \eqref{min copula for odd dimension} would only lead to a coskewness equal to $4.79$ resp.\ $0.21$.

\section{Conclusion} \label{conclusion}
In this paper, we find new bounds for the expectation of a product of random variables when marginal distribution functions are fixed but dependence is unknown. We solve this problem explicitly under some conditions on the marginal distributions and propose an algorithm to solve the problem in the general case. We introduce the novel notion of  standardized rank coskewness, which unlike coskewness,  is unaffected by marginal distributions and thus  appears useful for better understanding the degree of coskewness that exists among three random variables. 
\begin{doublespace}
\footnotesize
\paragraph{Acknowledgments:}
The authors thank two reviewers for their valuable hints and remarks which helped to improve the paper. The authors gratefully acknowledge funding from Fonds Wetenschappelijk Onderzoek (grants FWOAL942 and FWOSB73). Jinghui Chen would like to thank Xin Liu for his comments on the paper. 

\bibliography{ARXIV}
\end{doublespace}

\end{document}